\documentclass{article}
\usepackage{amscd,amssymb}
\usepackage{amsfonts}
\usepackage{hyperref}
\usepackage{color}
\usepackage[portuguese]{babel}
\usepackage[ansinew]{inputenc}
\usepackage{bbm}

\newtheorem{theorem}{Theorem}

\newtheorem{definition}[theorem]{Definition}

\newtheorem{lemma}[theorem]{Lemma}

\newtheorem{proposition}[theorem]{Proposition}

\newenvironment{proof}[1][Proof]{\noindent\textbf{#1.} }{\ \rule{0.5em}{0.5em}}

\begin{document}

\title{Orbit equivalence of linear systems on manifolds and semigroup actions on homogeneous spaces\thanks{This work was partially supported by
CNPq/Universal grant n$^{\circ }$ 476024/2012-9}}
\author{R. M. Hungaro\\Universidade Estadual de Maring\'{a} - Brazil, \and  O.G. Rocio
\\Universidade Estadual de
Maring\'{a} - Brazil, \and A. J. Santana\thanks{Partially supported
by Funda\c{c}\~{a}o Arauc\'{a}ria grant n$^{\circ }$
$20134003$}\\Universidade Estadual de Maring\'{a} - Brazil}

\maketitle

\begin{abstract}
 In this paper we introduce the notion  of orbit equivalence for semigroup actions
 and the concept of generalized linear control system on  smooth manifold.   The main
goal   is to prove that, under certain conditions, the semigroup
system  of a generalized linear control system on a smooth manifold
is orbit equivalent to the semigroup system of a linear control
system on a homogeneous space.
\end{abstract}

\noindent {}\noindent \textit{AMS 2010 subject classification}:
20M99, 37A20, 57S25, 93B05, 93B99, 93C99 \newline \noindent

\section{Introduction}

Although the control theory originated about a century ago,  there
is no global theory yet with general hypothesis. However, in special
cases, the study of control theory have made rapid progress in the
last decades. For example, the control theory on Lie groups has
achieved significant advances due especially its relationship with
the actions of semigroups on Lie groups, implying in good results in
the study of control sets and controllability (see e.g. Elliott
\cite{elliott}, Jurdjevic \cite{jurdjevic}, Rocio, San Martin and
Santana \cite{RoSaSacone}, Rocio, Santana and Verdi \cite{RoSaVe}
and Sachkov \cite{sachkov}).

Until the 1990s the theory of control systems on Lie groups was
restricted, basically, to the  control system of invariant vector
fields. In Ayala and Tirao \cite{ayti}, this study was expanded with
the introduction of linear control systems on Lie groups and
developed rapidly in recent years, the first papers on this subject
concern  about controllability (see e.g. \cite{aysm}, \cite{ayti}
and \cite{markus}).  In Jouan \cite{jouan1}, considering a control
system on a manifold given by complete linear vector fields that
generate a finite dimensional Lie algebra, it was showed a
equivalence between this system and a linear control system on
homogeneous space.

In our paper, initially we formalize the notion of orbit equivalence
for semigroups actions on manifolds. Then we establishes conditions
for an action of a semigroup of a control system on a manifold is
orbit equivalent to the action of a semigroup in a homogeneous
space. In the sequence, we introduce the concept of linear control
system on manifold, called generalized linear control system. The
main result of this paper establishes conditions under which the
action of the semigroup associated to the generalized linear control
system is orbit equivalent to the action of a semigroup on a
homogeneous space.

We  now touch some    control theoretic aspects related with our
work. Consider $G$   a (finite dimensional) connected and simply
connected Lie group. Suppose that $G$ acts  transitively on a
manifold $M$ and take $H$ a closed subgroup. Let $\pi:G \rightarrow
G/H$ be the canonical projection. A linear control system on $G/H$
is a special case of control systems  where the drift is
$\pi$-related with a linear vector field on $G$ and the controlled
vector fields are
projections of right invariant vector fields on $G$. Take  $\mathfrak{g%
}$ the Lie algebra    given by the right invariant vector fields on
$G$.
Using the same notations of Ayala and San Martin \cite{aysm} and \cite{ayti}   and denoting by $%
e$ the identity of $G$, a vector field ${\mathcal{X}}$ on $G$ is
called \emph{linear} if for all $Y\in \mathfrak{g}$ we have
$[{\mathcal{X}},Y]\in \mathfrak{g}$ and ${\mathcal{X}}(e)=0$. Hence
a \emph{linear control system} on $G$ is defined as
\[
\dot{x}={\mathcal{X}}(x)+\sum_{j=1}^{m}u_{j}Y_{j}(x),
\]
where the drift ${\mathcal{X}}$ is a linear vector field on $G$,
$Y_{1}, Y_{2}, \ldots , Y_{m} \in \mathfrak{g}$ and $u=\{u_{1},
u_{2}, \ldots , u_{m} \} \in  \mathbb{R}^{n}$.

We recall the definition of linear control system on $G/H$ given  in
\cite{jouan1}. A vector field on $G/H$ is called \emph{invariant} if
it is the ${\pi}_{\ast}$-image of some  right invariant vector field
on $G$ and is called \emph{linear} if is $\pi $-related with a
linear vector field on $G$. Hence if the drift of a control system
on $G/H$ is linear and the controlled vector fields are invariant
then the system is called \emph{linear control system} on
$\frac{G}{H}$.

In this direction, the main concept of our paper is the generalized
linear control system on manifolds. Take $\mathcal{L} (TM)$ the Lie
algebra of the differentiable vector fields on  $M$. A
\emph{generalized linear control system} on $M$ is a control system
 \[ \dot{x}=\mathcal{F} (x)
 +\sum_{j=1}^{m}u_{j}Y_{j}(x) , \] where
$\mathcal{F}, Y_{j} \in \mathcal{L} (TM)$ for every $j=1, \dots ,
m$, $\Gamma=\{Y_{1},\ldots ,Y_{m}\}$ generates a finite dimensional
Lie subalgebra $\mathcal{L}(\Gamma)$ of $\mathcal{L}(TM)$, every
vector field $Y_{i}\in \Gamma$ is complete,
 $[\mathcal{F},X]\in \mathcal{L} (\Gamma)$ for all $X\in \mathcal{L} (\Gamma)$
 and there exists  $x_{0}\in M$ such that
$\mathcal{F}_{x_{0}}=0$. The motivation to study these systems come
from the need to formalize concepts involved in control theory on
manifolds that transfer several issues, such as the controllability,
 to be treated in more pleasant state space such as  Lie groups.

About the structure of this paper, in the second section we
introduce the notion of orbit equivalence and topological conjugacy
for semigroup actions and give some properties related with control
sets. In the third section we fix the control theoretic notations
and relates state equivalent control systems with diffeomorphic
control systems. In the fourth section we prove that given a control
system on $M$, the semigroup system on $M$ is orbit equivalent to a
semigroup action on a homogeneous space. In the last section we
prove our main result which states that supposing $\Gamma$  is
transitive on $M$ and taking $G$ the connected and simply connected
Lie group with Lie algebra $\mathcal{L} (\Gamma)$, then the
semigroup system of the above system is orbit equivalent to a
semigroup system of a linear control system on a $G$-homogeneous
space.

\section{Orbit equivalence}

In this section, we define the notions of orbit equivalence for
semigroups actions and topological conjugacy for skew product, this
concepts will be necessary in the next sections. We establish some
relations between orbit equivalence and control sets for semigroups
actions. We begin recalling some concepts of the theory of control
sets (for more details see e.g. San Martin \cite{san martin} and San
Martin and Tonelli \cite{san martin e tonelli}).Take a non empty
semigroup $S$ acting on a topological space $M$. The semigroup $S$
is said to be accessible if ${\rm int}Sx \neq \emptyset$ for every
$x \in M$. A control set for the $S$-action on $M$ is a subset $C
\subset M$ such that ${\rm int}C \neq \emptyset$, $C \subset {\rm
cl}(Sx)$ for all $x \in C$ and $C$ is maximal with the first two
properties. If ${\rm cl}C = {\rm cl}(Sx)$ for all $x \in C$,  the
control set $C$ will be named invariant control set. We also recall
the partial ordering between control sets  given by $C_{1} < C_{2}$
if there exists $x \in C_{1}$ such that ${\rm cl}(Sx) \cap C_{2}
\neq \emptyset$.

Now about equivalence of semigroups we have the following
definitions.

\begin{definition}
Let $M_{1}$ and $M_{2}$ be topological spaces. Consider $S$ and $T$
semigroups. The actions $(M_{1},S)$ and $(M_{2},T)$ are called orbit
equivalent, if there exists an homeomorphism $f:M_{1} \rightarrow
M_{2}$ such that $f(Sx)=Tf(x)$ for all $x \in M_{1}$. The map $f$ is
called orbit equivalence map.
\end{definition}

Some authors call the pair $(M,S)$ as transformation semigroup (see
e.g. Ellis in \cite{ellis} and Sousa in \cite{josiney})

 Locally, we have that the actions $(M_{1},S)$ and
$(M_{2},T)$ are called orbit equivalent restricted to a subset $C
\subset M_{1}$ if there exists an homeomorphism $f:M_{1} \rightarrow
M_{2}$ such that $f(Sx)=Tf(x)$ for all $x \in C$.

Now supposing the existence of control sets we give some properties
of orbit equivalent actions. Recall that taking the topological
space as flag manifolds, there exist always control sets (see e.g.
\cite{san martin} and \cite{san martin e tonelli}).

\begin{proposition}
\label{preservecs} Suppose that  $(M_{1},S)$ and $(M_{2},T)$ are
orbit equivalent. Hence if $C_{S}$ is a control set for $S$ then
$f(C_{S})$ is a control set for $T$ in $M_{2}$. On the other hand,
if $C_{T}$ is a control set for $T$ in $M_{2}$ then $f^{-1}(C_{T})$
is a control set for $S$ in $M_{1}$ .
\end{proposition}

\begin{proof}
Note that $\mathrm{int}(f(C_{S})) \neq \emptyset$ and $f(C_{S})
\subset \mathrm{fe}(Ty)$ for every $y \in
f(C_{S})$. By hypotheses it follows $\mathrm{fe}%
(f(Sx)) \subset \mathrm{fe}(Tf(x))$, for all $x \in C_{S}$. The
proof of the converse is analogous.
\end{proof}

Moreover, the orbit equivalence preserves the order of control sets.

\begin{proposition}
The  topological conjugacy preserves the order of control sets.
\end{proposition}
\begin{proof}
Take the $S$ control sets $C_{1}$ and $C_{2}$. Suppose that $C_{1} <
C_{2}$, then there exists $x \in C_{1}$ such that ${\rm fe}(Sx) \cap
C_{2} \neq \emptyset$. Take $f:M_{1} \rightarrow M_{2}$ a
topological conjugace of the actions $(M_{1},S)$ and  $(M_{2},T)$
and consider
 the control sets $f(C_{1})$ and $f(C_{2})$ for the $T$ action. Take
 $f(x) \in f(C_{1})$ then ${\rm fe}(Tf(x)) \cap f(C_{2})={\rm fe}(f(Sx))
 \cap f(C_{2}).$
 But $\emptyset \neq f({\rm fe}(Sx)  \cap C_{2}) \subset f({\rm fe}(Sx))  \cap f(C_{2}) \subset {\rm fe}(f(Sx))
 \cap f(C_{2}).$
 Then ${\rm fe}(Tf(x)) \cap f(C_{2}) \neq \emptyset$.
\end{proof}

For the next proposition we recall the definition of the set of
transitivity $C^{0}$ of a control set $C$: $C^{0}=\{ x \in C : x \in
({\rm int}S)x \}$. It holds $C^{0}= ({\rm int}S)x$ for all $x \in
C^{0}$ (see \cite{san martin}).

\begin{proposition}
With the same notations, suppose that there exists a homeomorphism $f:M_{1}
\rightarrow M_{2}$ that send set of transitivity in set of transitivity,
that is, if $C \subset M_{1}$ is the $S$ invariant control set and $C^{0}$
its set of transitivity then $f(C)$ is the invariant control set for $T$
with $f(C^{0})$ its set of transitivity. Suppose also that $S$ and $T$ are
accessible. With this hypotheses we have $(M_{1},\mathrm{int}S)$ and $(M_{2},%
\mathrm{int}T)$ are orbit equivalent restricted to $C$.
\end{proposition}

\begin{proof}
Take $x \in C_{0}$ and $a \in \mathrm{int}S$, then $%
f(ax) \in f(C_{0})=(\mathrm{int}T)y$ for all $y \in f(C_{0})$, in
particular to $y=f(x)$. Then $f((\mathrm{int}S)x) \subset
(\mathrm{int}T)f(x)$, for all $x \in C_{0}$. It easy to prove that  $(\mathrm{int}%
T)f(x) \subset f((\mathrm{int}S)x)$. Now take $x \in
C=\mathrm{fe}C_{0}$, then exists a sequence $x_{n} \in C_{0}$ such
that $x_{n}$ converge to $x$. Moreover, we have
$f((\mathrm{int}S)x_{n}) = (\mathrm{int}T)f(x_{n}) \mbox{ for all }
n. $ As $f$ is homeomorphism it follows that $f(x_{n})$  converge to
$f(x).$

It follows that for all $g \in \mathrm{int}S$, exists $h \in
\mathrm{int}T$ such that $f(x_{n})=hf(x_{n})$. Hence, we have that
 for all $g \in \mathrm{int}S$ exists $h \in \mathrm{int}T$
 such that $f(gx)=hf(x)$.

Analogously,  taking $x \in C$ then
 for all  $h \in \mathrm{int}T$ exists $g \in \mathrm{int}S $
  such that  $ hf(x)=f(gx)$. Hence $f((\mathrm{int}S)x)=(\mathrm{int}%
T)f(x)$, for all $x \in C$.
\end{proof}

It is not difficult to prove this kind of converse:

\begin{proposition}
Consider the notations and assumptions as above. Suppose that $(M_{1},%
\mathrm{int}S)$ and $(M_{2},\mathrm{int}T)$ are orbit equivalent. Then $%
f(C^{0})=(f(C))^{0}$.
\end{proposition}

To finish this section, we establish a relation between the concepts
of conjugation and orbit equivalence.

Now suppose that  $S$ and $T$ have the identities $e_{S}$ and
$e_{T}$. Let $\varphi $ be a \emph{cocycle} on $X$ to $T$, that is,
$\varphi :S\times X\rightarrow T$ continuous with
\begin{eqnarray*}
\varphi \left( st,x\right) &=&\varphi \left( s,tx\right) \varphi
\left(
t,x\right) \quad \mbox{for all } s,t\in S, x\in X \mbox{, and} \\
\varphi \left( e_{S},x\right) &=&e_{T}\quad \mbox{for all }x\in X.
\end{eqnarray*}%
The cocycle property is appropriate to define the \emph{skew-product
transformation semigroup} on the product space $X\times Y$ given by
the mapping
\begin{equation}
\Phi :S\times X\times Y\rightarrow X\times Y,\quad \Phi \left(
s,x,y\right) =\left( sx,\varphi \left( s,x\right) y\right) .
\end{equation}%
We might write $s\left( x,y\right) $ instead of $\Phi \left(
s,x,y\right) $.

 We define the following subsemigroup of $T$, called \emph{system semigroup},
\begin{equation}
\mathbf{\mathcal{S}}=\left\{ \varphi \left( s_{n},x_{n}\right)
\varphi \left( s_{n-1},x_{n-1}\right) \cdots \varphi \left(
s_{1},x_{1}\right) :s_{j}\in S,x_{j}\in X,n\in \mathbb{N}\right\} .
\end{equation}%

By considering the action $\sigma $ restricted to the product $\mathbf{\mathcal{S}}%
_{\alpha }\times Y$, we have the transformation semigroup $\left(
\mathbf{\mathcal{S}},Y,\sigma \right) $ associated to the
skew-product transformation semigroup $\left( S,X\times Y,\Phi
\right) $.

To introduce the concepts of topological conjugacy and state
equivalence we consider, for $i=1,2$, the following two skew-product
transformation semigroups

\begin{center} $\Phi^{i} :S\times X^{i}\times Y^{i}\rightarrow
X^{i}\times Y^{i},\quad \Phi^{i} \left( s,x,y\right) =\left(
sx,\varphi^{i} \left( s,x\right) y\right)$ \end{center}

\begin{definition} Let $\xi:Y^{1}\rightarrow
Y^{2}$ and $\iota:X^{1} \rightarrow X^{2}$
 be maps such that  $\xi$ is continuous and satisfy:
\begin{center}$\xi(\varphi^{1}(s,x)y)=\varphi^{2}(s,\iota(x)) \xi(y)$, for all  $(s,x,y)\in S\times X^{1}\times Y^{1}$.\end{center}
In this case, we say that the skew product  $\Phi^{1}$ is
topologically semi conjugate to  $\Phi^{2}$. If $\xi$ is a
homeomorphism and $\iota$ is invertible, then the skew products are
called topologically conjugate.
\end{definition}

 In the particular case where   $\Phi^{1}$ and
 $\Phi^{2}$ are topologically  conjugate, $\iota = id_{X}$ and $\xi$ is a diffeomorphism, we
 say that  $\Phi^{1}$ and  $\Phi^{2}$ are \emph{state equivalent}.
 This terminology is inspired by the concept of state equivalence of
 control systems  (for more details see Agrachev and Sachkov in \cite{as}).

Now we prove a result that relates the concepts of conjugation and
orbit equivalence.
\begin{proposition}\label{conju} If $\Phi^{1}$ and $\Phi^{2}$
are topologically  conjugate then the actions
$(Y^{1},\mathcal{S}^{1})$ and $(Y^{2},\mathcal{S}^{2})$ are orbit
equivalent, where $\mathcal{S}^{1}$ and $\mathcal{S}^{2}$ are the
semigroup system of $\Phi ^{1}$ and $\Phi^{2}$ respectively.
\end{proposition}
\begin{proof}
 By hypothesis, there exists a homeomorphism
$\xi:Y^{1}\rightarrow Y^{2}$
 and an invertible map $\iota:X^{1}\rightarrow X^{2}$
such that $\xi(\varphi^{1}(s,x)y)=\varphi^{2}(s,\iota(x))\xi(y)$,
  for all   $(s,x,y)\in S\times X^{1}\times Y^{1}.$

Consider the following semigroups associated to $\Phi^{i}$ for
$i=1,2$
\begin{eqnarray}\mathcal{S}^{i}=\{\varphi^{i}(s_{n},x_{n}) \cdots \varphi^{i}(s_{1},x_{1});s_{j}\in S, x_{j}\in X^{i},n\in
\mathbb{N}\}.\nonumber\end{eqnarray}

Define the homeomorphism  $h$ as $\xi$. Then given  $a\in
h(\mathcal{S}^{1}y)$, we have $a=h(b)$, where $b\in
\mathcal{S}^{1}y$, i.e.,$b=\varphi^{1}(s_{n},x_{n}) \cdots
\varphi^{1}(s_{1},x_{1})y=\varphi^{1}(s_{n} \cdots s_{1},x)y.$ Hence
$a\in \mathcal{S}^{2}h(y)$, in fact, $a=\xi(\varphi^{1}(s_{n} \cdots
s_{1},x)y)=\varphi^{2}(s_{n} \cdots s_{1},\iota(x))h(y).$

For the opposite inclusion, consider $a\in \mathcal{S}^{2}h(y)$,
then $a=bh(y)$, with $b\in \mathcal{S}^{2}$, hence
$b=\varphi^{2}(s_{m},v_{m}) \cdots
\varphi^{2}(s_{1},v_{1})=\varphi^{2}(s_{m} \cdots s_{1},v). $ Then,
using a similar idea as above we prove that $a\in
h(\mathcal{S}^{1}y)$.
\end{proof}

\section{Conjugacy and state equivalence of control systems}

In this section we prove that if two systems are diffeomorphic then
they are state equivalent

 Take $M$ a differentiable and connected
$d$-dimensional manifold. Consider in $M$ the following control
system
\[
(\Sigma )\ \ \ \ \ \ \ \ \ \ \dot{x}(t)=X_{0}(x(t))+\sum_{j=1}^m
u_{j}X_{j}(x(t)) ,
\]
where $u: {\Bbb R} \rightarrow U$ is a piecewise constant map with
$U \subset {\Bbb R}^{n}$  compact and convex, and $X_{i}$ are
differentiable vector fields on $M$. Denote by $\mathcal{U}$ the set
of the maps $u$. It is well known that $\mathcal{U}$ is a metric
space (see e.g. Colonius and Kliemann in \cite{ColK00}). We assume
that for each $u$ and $x \in M$ this system has a unique solution
$\phi (t,u,x), t \in {\Bbb R}$, with $\phi (0,u,x)=x$.

As defined in \cite{ColK00}, take
\[ \Phi : {\Bbb R} \times \mathcal{U} \times M \rightarrow \mathcal{U} \times
M,   \Phi (t,u,x)=(\Theta _{t}(u), \phi (t,u,x)), \]
 the control flow of
this system, we know that it is a special case of skew-product
transformation semigroup (see   \cite{josiney}).

Then as a consequence of the previous theorem we consider two
control systems $\Sigma_{1}$ and $\Sigma_{2}$ as above, take their
control flows $\Phi _{1}$ and $\Phi _{2}$  and their correspondent
system semigroups $S_{\Sigma_{1}}$ and $S_{\Sigma_{2}}$. Now we
recall the construction of these semigroups, take the map $\varphi
_{t_{1}}^{u_{1}}:M_{1} \rightarrow M_{2}$ given by $\varphi
_{t_{1}}^{u_{1}} (x)=\varphi (t_{1},u_{1},x)$ then we have that
$\Sigma_{1}$ is a semigroup of diffeormophisms of $M_{1}$ given by
\begin{eqnarray}S_{\Sigma _{1}}=\{\varphi _{t_{r}}^{u_{r}}\circ \cdots \circ
\varphi _{t_{1}}^{u_{1}};u_{i}\in \mathcal{U},t_{i}\geq 0,r\in
\mathbb{N}\} . \nonumber\end{eqnarray}

 The  natural action of $S_{\Sigma_{1}}$ on  $M_{1}$ is defined as
$\varphi \cdot x=\varphi (x)$. In the same way we have  the
semigroup $S_{\Sigma _{2}}$. Recall that $(\Sigma _{1})$ and
$(\Sigma _{2})$ are called \textit{topologically conjugate} if there
exist a homeomorphism $\xi :M_{1}\rightarrow M_{2}$ and an
invertible map $\iota:\mathcal{U}\rightarrow \mathcal{V}$ such that
 $\xi (\varphi (t,u,x))=\psi (t,\iota (u),\xi (x))$, for all
  $(t,u,x)\in \mathbb{R}\times \mathcal{U}\times M_{1}$. Then as
  a consequence of Proposition \ref{conju} we   have the
  following proposition:

\begin{proposition}\label{general}
Suppose that $\Sigma_{1}$ and $\Sigma_{2}$ are topologically
conjugate then the actions $(M_{1},S_{\Sigma_{1}})$ and
$(M_{2},S_{\Sigma_{2}})$ are orbit equivalent.
\end{proposition}

Another important concept used to classify control systems is the
notion of  state equivalence, as defined in the previous section, in
this case $\xi $ is a diffeomorphism and  $\mathcal{U}=\mathcal{V}$.
This concept is used to classify control systems preserving
differentiable properties. A sufficient condition to guarantee that
$(\Sigma_{1})$ and $(\Sigma_{2})$ be state equivalent is the
existence of a diffeomorphism from $M_{1}$ to $M_{2}$ that preserves
the control systems. Precisely, suppose that
$\mathcal{U}=\mathcal{V}$ and that $\xi :M_{1}\rightarrow M_{2}$ be
a diffeomorphism. For each $u\in \mathcal{U}$ consider the vector
fields $Z_{u}$ in $M_{1}$ and $W_{u}$ in $M_{2}$ given by
\[
Z_{u}(x) =X_{0}(x)+\sum_{j=1}^{m}u_{j}X_{j}(x) \] and \[ W_{u}(\xi
(x)) =Y_{0}(\xi (x))+\sum_{j=1}^{m}u_{j}Y_{j}(\xi (x)) ,
\]
where $x\in M_{1}$. In this conditions we have:
\begin{proposition}\label{scs}
If $\xi :M_{1}\rightarrow M_{2}$ is a diffeomorphism such that
$\xi_{\ast}(Z_{u}(x))_{x}=W_{u}(\xi (x))$, for all $u\in
\mathcal{U}$ \ and $x\in M_{1}$ then the control systems  $(\Sigma
_{1})$ and $(\Sigma _{2})$ are state equivalent.
\end{proposition}
\begin{proof}
Given $u\in \mathcal{U}$ and $x\in M_{1}$ denote by $\varphi
(t,u,x)$ the unique solution of the system  $(\Sigma _{1})$ such
that $\varphi (0,u,x)=x$ and by  $\psi (t,u,\xi (x))$ the unique
solution of $(\Sigma _{2})$ such that  $\psi (0,u,\xi(x))=\xi(x)$.
Them $\frac{d}{dt}\varphi (t,u,x)=Z_{u}(\varphi(t,u,x))$, for all
$t\in \mathbb{R}$ and hence $\frac{d}{dt}\xi (\varphi
(t,u,x))=(\xi_{\ast})_{\varphi (t,u,x)}\frac{d}{dt}\varphi
(t,u,x)=\xi_{\ast}(Z_{u}(\varphi (t,u,x))_{\varphi
(t,u,x)}=W_{u}(\xi(\varphi (t,u,x)))$, showing that $\xi (\varphi
(t,u,x))$ is also the solution of the differential equation
$\dot{y}(t)=Y_{0}(y(t))+\sum_{j=1}^{m}u_{j}Y_{j}(y(t))$ in $M_{2}$,
with initial value $\xi(\varphi (0,u,x))=\xi(x)$. Therefore
$\xi(\varphi (t,u,x))=\psi (t,u,\xi(x))$, for all $(t,u,x)\in
\mathbb{R}\times \mathcal{U}\times M_{1}$.
\end{proof}

Another well known concept (see e.g. \cite{jouan1}) is the
diffeomorphic control systems.
\begin{definition}
Using the above notations, the control systems  $(\Sigma _{1})$ and
$(\Sigma _{2})$ are diffeomorphic  if there exists a diffeomorphism
$\xi :M_{1}\rightarrow M_{2}$ such that $\xi_{\ast}(X_{i})=Y_{i}$
for $0\leq i\leq m $
\end{definition}

Then  we have the following result that relates state equivalent
control systems with diffeomorphic control systems
\begin{proposition}\label{state}
If the control systems $(\Sigma _{1})$ and $(\Sigma _{2})$ are
diffeomorphic then they are state equivalent.
\end{proposition}
\begin{proof}
Let $\xi :M_{1}\rightarrow M_{2}$ be a diffeomorphism such that
$\xi_{\ast}(X_{i})=Y_{i}$ for $0\leq i\leq m$. It is easy to see
that $\xi_{\ast}(Z_{u}(x))_{x}=W_{u}(\xi (x))$. Then, by Proposition
\ref{scs}, the control systems $(\Sigma _{1})$ and $(\Sigma _{2})$
are state equivalences.
\end{proof}

\section{Orbit equivalence of semigroup system on homogeneous space}

Take a control system $\Sigma$ on a manifold $M$. The purpose of
this section is to prove that $(M,S_{\Sigma})$ is orbit equivalent
to a semigroup action on a homogeneous space. The Lie-Palais theorem
is fundamental to obtain this result.

We begin  supposing that there exists a Lie group $G$ acting
transitively on $M$. In this case, $M$ is diffeomorphic to a
homogeneous space of $G$. From this   we prove that there exist a
control system $\tilde{\Sigma}$ on a homogeneous space of $G$ such
that $\Sigma$ be orbit equivalent to $\tilde{\Sigma}$. Then take the
control system
\begin{center}
$(\Sigma )\ \ \ \ \ \ \ \ \ \ \dot{x}(t)=X_{0}(x(t))+
\sum_{i=1}^{m}u_{i}X_{i}(x(t))$
\end{center}
on   $M$ with the same hypothesis of the previous section. Consider
the Lie algebra $\mathcal{L}(TM)$ of all vector fields on $M$ and
take its Lie algebra  $\mathcal{{L}}(\Gamma )$, generated by the set
of vector field   $\Gamma =\{X_{0},X_{1},\ldots ,X_{m}\}$. Supposing
that $\mathcal{{L}}(\Gamma )$ has finite dimension we take the
connected and simply connected Lie group  $G$ with Lie algebra
$\mathcal{{L}}(\Gamma )$. A natural way to define the action of $G$
on $M$ is given in the following way. Denote by $\Psi _{t}^{X}$ the
flow of  $X\in \mathcal{L}(\Gamma )$. As every $g\in G$ can be
written as $g=e^{t_{1}X_{i_{1}}}\cdots e^{t_{s}X_{i_{s}}}$, for some
$ t_{i_{1}},\ldots , t_{i_{s}}\in \mathbb{R}$ and $X_{i_{1}},\ldots
, X_{i_{s}}\in \mathcal{L}(\Gamma )$, we can try to define an action
$\phi:G\times M\rightarrow M$ by $\phi (g,x)=\Psi
_{t_{i_{1}}}^{X_{i_{1}}}\circ \cdots \circ \Psi
_{t_{i_{s}}}^{X_{i_{s}}}(x)$. The problem is that there is not just
one way to write  $g\in G$ as product of exponentials. But using
Lie-Palais theorem, we can guarantee that this definition does not
depend on this fact. Before, we define the concept of infinitesimal
action.

\begin{definition}
Let $\mathfrak{g}$ be a Lie algebra and take  $M$ a differentiable
manifold. An infinitesimal action of $\mathfrak{g}$ on $M$ is a
homeomorphism  $\theta:\mathfrak{g}\rightarrow \mathcal{L}(TM)$.
\end{definition}

It is easy see that a differentiable action  $\phi :G\times
M\rightarrow M$ induces an infinitesimal action $\theta
:\mathfrak{g} \rightarrow \mathcal{L}(TM)$, in fact, define $\theta
(X)(x)=d\phi _{x}\mid _{1}(X)$, where $x\in M$ and $1$ denote the
identity element of $G$. One kind of converse is the Lie-Palais
Theorem.

\begin{theorem}
\label{palais}[Lie-Palais] Let $\mathfrak{g}$ be a real and finite
dimensional Lie algebra. Take $G$ the connected and simply connected
Lie group with Lie algebra $\mathfrak{g}$. Consider  $\theta
:\mathfrak{g}\rightarrow \mathcal{L}(TM)$ an infinitesimal action of
 $\mathfrak{g}$ and suppose that the vector fields  $\theta (X),X\in
\mathfrak{g}$ be complete. Then exists a differentiable action
$\phi:G\times M\rightarrow M$ such that $\theta $ is the
correspondent infinitesimal action.
\end{theorem}

The proof of Lie-Palais theorem can be found in San Martin
\cite{SM2}.

\begin{proposition}\label{prop1}
Let $\Gamma =\{X_{0},X_{1},\l dots ,X_{m}\}$ be a family of complete
and differentiable vector fields on the  manifold  $M$ such that the
Lie algebra $\mathcal{L}(\Gamma)$ has finite dimension. Denote by
$G$ the connected and simply connected Lie group whose Lie algebra
is $\mathcal{L}(\Gamma)$. Then we can define the following action
  $\phi:G\times M\rightarrow M$. Take $g\in G$ hence
$g=e^{t_{1}X_{i_{1}}}\cdots e^{t_{s}X_{i_{s}}}$, for some
$t_{i_{1}},\ldots ,t_{i_{s}}\in \mathbb{R}$ and $X_{i_{1}},\ldots
,X_{i_{s}}\in \Gamma $. Therefore $ \phi (g,x)=\Psi
_{t_{i_{1}}}^{X_{i_{1}}}\circ \cdots \circ \Psi
_{t_{i_{s}}}^{X_{i_{s}}}(x)$.
\end{proposition}
\begin{proof}
Note that the inclusion map $\theta:\mathcal{L}(\Gamma)\rightarrow
\mathcal{L}(TM)$ is an infinitesimal action of the Lie algebra $
\mathcal{L}(\Gamma)$ on $M$. By Lie-Palais theorem, there exists a
differential action  $\phi:G\times M\rightarrow M$ such that
$X_{x}=d\phi _{x}\mid _{1}(X),\forall x\in M,\forall X\in
\mathcal{L}(\Gamma)$. From the description of this action, take
$X\in \mathcal{L}(\Gamma)$ and consider the field $ (X,\theta (X))$
on $G\times M$. The trajectory of this field beginning in  $(1,x)$
is $(e^{tX},\Psi _{t}^{X}(x)),$ that is, $$\frac{d}{dt}\mid
_{t=0}(e^{tX},\Psi _{t}^{X}(x))=(X,\theta (X)_{x}).$$ On the other
hand, taking  $\phi_{x}:G\rightarrow M$, applying in $e^{tX}(1)\in
G$ and using Lie-Palais Theorem we have also
\[
\frac{d}{dt}\mid_{t=0}(e^{tX},\phi _{x}(e^{tX}))=(X,\theta (X)_{x}).
\]

Then $(e^{tX},\Psi _{t}^{X}(x))=(e^{tX},\phi _{x}(e^{tX}))$, i.e.,
\begin{eqnarray}\label{8} \Psi _{t}^{X}(x)=\phi _{x}(e^{tX})=\phi
(e^{tX},x),\forall x\in M,\forall X\in \mathcal{L}(\Gamma ),\forall
t\in \mathbb{R}.\end{eqnarray} Hence, if $X,Y\in \mathcal{L}(\Gamma
),x\in M$ e $t,\tau \in \mathbb{R}$, then
\[
\Psi _{\tau }^{Y}(\Psi _{t}^{X}(x))=\phi _{\phi
_{x}(e^{tX})}(e^{\tau Y})=\phi _{e^{\tau Y}e^{tX}}(x).
\]

By induction we have for all   $x\in M$
  $X_{1},\ldots ,X_{n}\in \mathcal{L}(\Gamma)$ and $t_{1},\ldots ,t_{n}\in \mathbb{R}$ that $\Psi
_{t_{1}}^{X_{1}}\circ \cdots \circ\Psi _{t_{n}}^{X_{n}}(x)=\phi
_{e^{t_{1}X_{1\cdots }}e^{t_{n}X_{n}}}(x)$.
\end{proof}

If we suppose that the family  $\Gamma $ is transitive we have:

\begin{theorem}\label{difeo}
Let $\Gamma =\{X_{0},X_{1},\ldots ,X_{m}\}$ be a family of
transitive, complete and differentiable vector fields on the
connected manifold $M$. Suppose that the Lie algebra
$\mathcal{L}(\Gamma)$ has finite dimension and take $G$ its
associated connected and simply connected Lie group. Then, $M$ is
diffeomorphic to a $G$-homogeneous space.
\end{theorem}
\begin{proof}
Consider the action $\phi :G\times M\rightarrow M$ the action given
in the previous proposition. Then,
$\phi(g,x)=\Psi_{t_{i_{1}}}^{X_{i_{1}}}\circ \cdots \circ
\Psi_{t_{i_{s}}}^{X_{i_{s}}}(x)$ and as  $\Gamma $ is transitive we
have that this action is transitive. Hence, fixing $x_{0}\in M$ and
considering the isotropy subgroup $H_{x_{0}}=\{g\in G:\phi
(g,x_{0})=x_{0}\}$ we that  $M$ is diffeomorphic to the homogeneous
space $G/H_{x_{0}}$.
\end{proof}

 Now we describe this above diffeomorphism.  If $x\in M$, as $\Gamma $
 is transitive, there exist
$X_{i_{1}},\ldots ,X_{i_{s}}\in \Gamma$ and $t_{i_{1}},\ldots
,t_{i_{s}}\in \mathbb{R}$ such that $$x=\Psi
_{t_{i_{1}}}^{X_{i_{1}}}\circ \cdots \circ \Psi
_{t_{i_{s}}}^{X_{i_{s}}}(x_{0})=\phi_{e^{t_{i_{1}}X_{i_{1}}}\cdots
e^{t_{i_{s}}X_{i_{s}}}}(x_{0}).$$ In this case, the above
diffeomorphism, denoted by $\xi :M\longrightarrow G/H_{x_{0}}$, is
defined by $\xi (x)=(e^{t_{i_{1}}X_{i_{1}}}\cdots
e^{t_{i_{s}}X_{i_{s}}})H_{x_{0}}$ and its inverse is given in the
following way. Given $g\in G$, there exist $X_{i_{1}},\ldots
,X_{i_{s}}\in\Gamma $ and $ t_{i_{1}},\ldots ,t_{i_{s}}\in
\mathbb{R}$ such that $ g=e^{t_{i_{1}}X_{i_{1}}}\cdots
e^{t_{i_{s}}X_{i_{s}}}$. Remember that this choices are not unique.
In this case, define
\begin{eqnarray}\xi^{-1}(gH_{x_{0}})=\Psi _{t_{i_{1}}}^{X_{i_{1}}}\circ
\cdots \circ \Psi _{t_{i_{s}}}^{X_{i_{s}}}(x_{0}),\end{eqnarray}
where by Proposition \ref{prop1}, this definition does not depend on
the exponential form of $g$.

To finish this section we  prove a result that relates a control
system on $M$ with his induced system on $G/H_{x_{0}}$. But first we
show an important lemma to the sequence of this paper. Consider the
map   $f$ defined as $\xi^{-1}\circ \pi :G\longrightarrow M $, where
$\pi :G\longrightarrow G/H_{x_{0}}$ is the canonical projection.
With this, $\pi (g)=\xi (f(g)),\forall g\in G$, and as $\xi^{-1}$
and $\pi$ are surjective maps it follows that   $f$ is surjective.

\begin{lemma}\label{lema}
If $X\in \mathcal{L}(\Gamma )$ then $\pi _{\ast }(X)=\xi_{\ast
}(X)$.
\end{lemma}
\begin{proof}
Take $X\in \mathcal{L}(\Gamma)$, $g\in G$ and $x\in M$ such that
$f(g)=x$. Consider $e^{tX}g$ the trajectory of $X$ in $G$ with
initial point $g\in G$. Consider  $\Psi_{t}^{X}(x)$ the trajectory
of $X$ in $M$ with initial point $x\in M$. Then,
\begin{eqnarray}\label{6}\frac{d}{dt}|_{t=0}(\xi(\Psi_{t}^{X}(x)))=d\xi|_{x}(X_{x}) \mbox{ and }
\frac{d}{dt}|_{t=0}(\pi(e^{tX}g))=d\pi|_{g}(X_{g}).\end{eqnarray}

Note that there exist  $X_{i_{1}},\ldots ,X_{i_{s}}\in \Gamma$ and
$t_{i_{1}},\ldots ,t_{i_{s}}\in \mathbb{R}$ such that
$g=e^{t_{i_{1}}X_{i_{1}}}\cdots e^{t_{i_{s}}X_{i_{s}}}$. Also there
is  $g_{1}\in G$ such that $x=\phi(g_{1},x_{0})$. Analogously, there
are  $X_{j_{1}},\ldots ,X_{j_{k}}\in \Gamma$ and $t_{j_{1}},\ldots
,t_{j_{k}}\in \mathbb{R}$ such that
$g_{1}=e^{t_{j_{1}}X_{j_{1}}}\cdots e^{t_{j_{k}}X_{j_{k}}}$. Then
$gg_{1}=e^{t_{i_{1}}X_{i_{1}}}\cdots
e^{t_{i_{s}}X_{i_{s}}}e^{t_{j_{1}}X_{j_{1}}}\cdots
e^{t_{j_{k}}X_{j_{k}}}.$ Hence
\begin{eqnarray}\xi(\phi(g,x))&=&\xi(\phi(g,\phi(g_{1},x_{0})))=gg_{1}H_{x_{0}}.\nonumber
\end{eqnarray}

 As $\pi(gg_{1})=gg_{1}H_{x_{0}}$, then $\pi(gg_{1})=\xi(\phi(g,x)).$
In particular, given  $X\in \mathcal{L}(\Gamma)$ and $t\in
 \mathbb{R}$, $\pi(e^{tX}g)=\xi(\phi(e^{tX},x)).$
 By (\ref{8}), we have $\pi(e^{tX}g)=\xi(\Psi_{t}^{X}(x)).$

 Hence, from (\ref{6})  we have
$\pi_{\ast}(X)=\xi_{\ast}(X).$

\end{proof}

Returning to the control system $(\Sigma )$ on $M$ and taking the
vector fields $\widetilde{X}_{i}=\pi _{\ast }(X_{i}),0\leq i\leq m$
on $G/H_{x_{0}}$, we define the following control system on
$G/H_{x_{0}}$:
\begin{center}
$(\widetilde{\Sigma })\ \ \ \ \ \ \ \ \ \
\dot{\widetilde{x}}(t)=\widetilde{X}_{0}(
\widetilde{x}(t))+\sum_{i=1}^{m}u_{i}\widetilde{X}_{i}(\widetilde{x}(t))$.
\end{center}

 Note that by Lemma \ref{lema}, $\xi_{\ast}(X_{i})=\widetilde{X}_{i}$ for $0\leq i\leq m$, and knowing that
  $\xi :M\longrightarrow G/H_{x_{0}}$ is a diffeomorphism,
we have that the control systems $(\Sigma)$ and
$(\widetilde{\Sigma})$ are diffeomorphic. Consequently, by
Proposition \ref{state} it follows that $(\Sigma)$ and
$(\widetilde{\Sigma})$ are state equivalent. Denoting by
$S_{\Sigma}$ and $S_{\widetilde{\Sigma}}$ the associated semigroups,
using the Proposition \ref{general} and recalling that state
equivalent systems are topologically conjugate we conclude the
following theorem

\begin{theorem}\label{teo}
Let $M$ be a connected and differentiable manifold and consider
$(\Sigma)$ the above control system. Suppose that $\Gamma
=\{X_{0},X_{1},\ldots ,X_{m}\}$ is transitive and complete on $M$.
Suppose also that the Lie subalgebra of $\mathcal{L}(TM)$, generated
by  $\Gamma $, has finite dimension. Then, the action
$(M,S_{\Sigma})$ is orbit  equivalent to a semigroup action on a
homogeneous space.
\end{theorem}
\begin{proof}
As we see above, the action  $(M,S_{\Sigma })$ is orbit equivalent
to the action  $(G/H_{x_{0}},S_{\widetilde{\Sigma }}).$
\end{proof}

\section{Generalized linear system on manifolds}

Our goal in this section is to introduce the concept of linear
control systems on general manifolds and using the results of the
previous sections show that, under certain conditions, a linear
control system   on a manifold is orbit equivalent to a linear
control  system on a homogeneous space.

Recall that the concept of linear control system depends on the
structure of the Lie group. Then to define this concept on general
manifolds we must work around the lack of the Lie group. Now we
define the generalized linear control system. Let $M$ be a connected
manifold with finite dimension and denote by $\mathcal{L} (TM)$ the
Lie algebra of the differentiable vector fields on  $M$.

\begin{definition}
A generalized linear control system on $M$ is a control system
\begin{center}$ (\Lambda ) \ \ \ \ \ \ \ \ \ \ \ \ \ \
\dot{x}=\mathcal{F} (x)+\sum_{j=1}^{m}u_{j}Y_{j}(x)$\end{center}
where

\begin{enumerate}
\item the set of vector fields $\Gamma=\{Y_{1},\ldots ,Y_{m}\}$ generates the finite
dimensional  Lie subalgebra $\mathcal{L} (\Gamma)$ of
$\mathcal{L}(TM)$ and every vector field $Y_{i}\in \Gamma$ is
complete,

\item $\mathcal{F}\in\mathcal{L} (TM)$,  $[\mathcal{F},X]\in \mathcal{L} (\Gamma),\forall
X\in \mathcal{L} (\Gamma)$ and there exists  $x_{0}\in M$ such that
$\mathcal{F}_{x_{0}}=0$,

\item and $u=(u_{1},\ldots ,u_{m})\in\mathbb{R}^{m}$.
\end{enumerate}
\end{definition}

It is clear that a linear control system on a Lie group is a
generalized linear control system, but not all generalized linear
control system is a linear control system. In fact, in case of
generalized linear control system, the vector fields $Y_{i}$ are not
necessarily invariants.

Now we have our main result
\begin{theorem}
Consider $M$ a connected and simply connected differentiable
manifold. Let $(\Lambda )$ be the above generalized linear control
system on $M$. If $\Gamma=\{Y_{1},\ldots ,Y_{m}\}$ is transitive on
$M$, then the action $(M,S_{\Lambda })$ is orbit equivalent to a
semigroup action associated to a linear control system on a
homogeneous space.
\end{theorem}
\begin{proof} By Theorem \ref{teo} we need define a diffeomorphism  $\xi$
that carries  $\Lambda$ in a linear control system
$\widetilde{\Lambda}$ on a homogeneous space. Now we define this
homogeneous space, by Theorem \ref{difeo} we take $G$ the
  connected and simply connected Lie group with Lie algebra $\mathcal{L}(\Gamma)$.
  Note that $G$ acts transitively on $M$. From this action,
  take  $H\subset G$, the isotropy subgroup in
     $x_{0}\in M$, then we have the diffeomorphism
$\xi: M \rightarrow G/H$ given by $\xi (g \cdot x_{0})=gH$, where
$\cdot$ denotes the action of $G$ on $M$. Hence, we need to show
that when we apply  $\xi_{\ast}$ in $(\Lambda)$ we get a linear
control system on $\frac{G}{H}$, i.e., $\xi_{\ast}(\mathcal{F})$ is
a linear vector field and $\xi_{\ast}(Y_{j})$ is right invariant
vector field for $i=\{1, \ldots ,m\}$.

 As $\xi$ is a
diffeomorphism, then  $\xi_{\ast}(Y_{j})$ and $Y_{j}$ are
$\xi$-related. Then, as $Y_{j}$ is invariant we have that
${\pi}_{\ast}(Y_{j})$ is invariant on  $G/H$. Moreover, by Lemma
\ref{lema} we have that ${\pi}_{\ast}(Y_{j})={\xi}_{\ast}(Y_{j})$,
for all $X\in\mathfrak{g}$, therefore $\xi_{\ast}(Y_{j})$ is
invariant on $G/H$.

We need to show that $\xi_{\ast}(\mathcal{F})$ is a linear vector
field, i.e., $\xi_{\ast }(\mathcal{F})$ is $\pi$-related with a
linear vector field on $G$. First, we find this linear vector field
on $G$. By Lemma \ref{lema}, if $X\in \mathfrak{g}$ then
\begin{eqnarray}\label{0}[\xi_{\ast }(\mathcal{F}),\pi_{\ast
}(X\mathcal{})]=\pi_{\ast}[\mathcal{F} ,X]. \end{eqnarray}

Let $D:\mathfrak{g}\longrightarrow \mathfrak{g}$ be a derivation
defined by  $D(X)=[\mathcal{F},X]$. As $G$ is connected and simply
connected,  there exists a linear vector field $\mathcal{X}$ on $G$
such that $D(X)=[\mathcal{X},X],\ \forall X\in \mathfrak{g}.$

Then we prove that $\xi_{\ast}(\mathcal{F})$ is $\pi-$related with
$\mathcal{X}$. To do this, we prove that $\pi_{\ast}(\mathcal{X})$
is $\pi-$related with $\mathcal{X}$ and then we show that
$\pi_{\ast}(\mathcal{X})=\xi_{\ast}(\mathcal{F})$.

 Hence we first show that $H$ is invariant by the flow
  ${\phi}_{t}$ of $\mathcal{X}$. Note that the vector field
   $\xi_{\ast}(\mathcal{F})$ in the point $H \in
G/H$, $\xi_{\ast}(\mathcal{F})_{H}$,
 is equal to \begin{eqnarray}\label{4}d\xi \mid _{x_{0}}(\mathcal{F}_{x_{0}})=d\xi \mid
_{x_{0}}(0)=0,\end{eqnarray} since $\xi (x_{0})=\xi (1 \cdot
x_{0})=1H=H$.

Note also that, $\pi_{\ast}(Y)_{H}=0$ for all $Y$ in the Lie algebra
$\mathfrak{h}$ of $H$. In fact, as $Y\in \mathfrak{h}$  then $\exp
(tY)\in H$, for all $t\in \mathbb{R}$. So, $\exp (tY) \cdot
x_{0}=x_{0}$, for all $t\in \mathbb{R}$. Hence, $\pi _{\ast
}(Y)_{H}=\frac{d}{dt}\mid _{t=0}(\exp (tY) \cdot H)=\frac{d}{dt}\mid
_{t=0}(H)=0$. Therefore, as $ \pi _{\ast }(Y)_{H}=0$ and $\xi _{\ast
}(\mathcal{F})_{H}=0$, we have that
\begin{eqnarray} [\xi _{\ast }(\mathcal{F}),\pi _{\ast
}(Y)]_{H}=0, \forall Y\in \mathfrak{h}.\nonumber\end{eqnarray}

Note that $\pi _{\ast }[\mathcal{X},Y]_{H}=0$. Hence, its flow given
by $gH \mapsto ({\rm exp}t[\mathcal{X},Y]) gH$ satisfies $({\rm
exp}t[\mathcal{X},Y]) \cdot H=H$, for all $t\in \mathbb{R}$.
Therefore, ${\rm exp}t[\mathcal{X},Y] \in H$, then
$D(Y)=[\mathcal{X},Y]\in \mathfrak{h}, \forall Y\in \mathfrak{h}.$

This implies  that
 \[ \phi _{t}(\exp Y)=\exp (e^{tD}Y)= \exp (I + tD + \frac{t^{2}D^{2}}{2!} + \cdots )Y  \in H . \]

 Then,  $\phi _{t}(\exp Y)\in H$, $\forall t\in
\mathbb{R}$ e $\forall Y\in \mathfrak{h}$. As $M$ is connected,
simply connected and diffeomorphic to $\frac{G}{H}$, it follows that
$\frac{G}{H}$ is simply connected. Then $H$ is connected. Hence,
every element of  $H$ is product of exponentials of elements of
$\mathfrak{h}$ and as  $\phi _{t}$  is an isomorphism then $H$ is
invariant by the flow $\phi _{t}$. Consequently,
 $\pi _{\ast }(\mathcal{X})$ is a vector field on $G/H$ $\pi $-related with  $
\mathcal{X}$.

 To conclude the proof, we show that  $\pi _{\ast }(\mathcal{X})=\xi _{\ast }(
\mathcal{F})$. In fact, if  $X\in \mathfrak{g}$, then $[\xi _{\ast
}(\mathcal{F}),\pi _{\ast }(X)]= \pi _{\ast}[\mathcal{X}, X].$ Note
that  $[\pi _{\ast}(\mathcal{X}),\pi _{\ast}(X)]$ and
$[\mathcal{X},X]$ are $\pi$-related, hence  $\pi
_{\ast}[\mathcal{X}, X]=[\pi _{\ast }( \mathcal{X}),\pi _{\ast
}(X)]$, therefore $[\pi _{\ast }(\mathcal{X})-\xi _{\ast
}(\mathcal{F}),\pi _{\ast }(X)]=0,\forall X\in \mathfrak{g} .$

Then the flow of  $\pi _{\ast }( \mathcal{X})-\xi _{\ast
}(\mathcal{F})$ on $G/H$, denoted by $ \alpha _{t}$, commute with
the flow of $\pi _{\ast }(X)$, given by $gH \mapsto ({\rm exp}tX)
gH$.

As  $\mathcal{X}$  is linear, then $\pi _{\ast
}(\mathcal{X)}_{H}=0$. Moreover, from (\ref{4}) we have
$\xi_{\ast}(\mathcal{F})_{H}=0$, then $\pi _{\ast
}(\mathcal{X})_{H}=\xi _{\ast }(\mathcal{F})_{H}=0$. Hence, $(\pi
_{\ast }(\mathcal{X})-\xi _{\ast }(\mathcal{F}))_{H}=0$, so $\alpha
_{t}(H)=H,$  $\forall t\in \mathbb{R}$.

Consider, $g\in G$, as $G$ is connected, there exist
$Y_{i_{1}},\ldots ,Y_{i_{r}}\in\mathfrak{g}$ e $t_{i_{1}},\ldots
,t_{i_{r}}\in\mathbb{R}$ such that $g=\exp
(t_{i_{1}}Y_{i_{1}})\cdots \exp (t_{i_{r}}Y_{i_{r}})$. Then
\[
\alpha _{t}(gH)=gH, \forall t\in \mathbb{R}.
\]

 Therefore, $(\pi _{\ast }(\mathcal{X})-\xi _{\ast }(\mathcal{F}))_{gH}=0$,
i.e., $\pi _{\ast }(\mathcal{X})=\xi _{\ast }(\mathcal{F})$.
\end{proof}

In the previous theorem the hypothesis $M$ simply connected is
fundamental. In fact, this implies that $H$ is connected. Then,
consequently every element of $H$ is a product of exponentials of
elements of $\mathfrak{h}$. Now we show that a generalization of
this last theorem, where it is not necessary has  $M$ simply
connected. To get this, the concept of universal covering is
essential. Then, recall that given a universal covering $f:\tilde{M}
\rightarrow M$, where $\tilde{M}$ is a differential manifold such
that  $f$ is differentiable. We can lift the vector fields on $M$ to
$\tilde{M}$, that is, given  $Z \in TM$, the vector field $\tilde{Z}
\in T\tilde{M}$ is defined in the following way. Given $\tilde{x}
\in \tilde{M}$, as $f$ is differentiable covering, there exist open
neighborhoods    $\tilde{U}$ of $\tilde{x}$ in $\tilde{M}$ and $U$
of $x$ in $M$ such that $f{\mid}_{\tilde{U}}: \tilde{U} \rightarrow
U$ is diffeomorphism. Then we define
\[
\tilde{Z}_{\tilde{x}}=d(f{\mid}_{\tilde{U}})^{-1}{\mid}_{x}(Z_{x}).
\]

Consider  $M$ a differentiable connected manifold and take in $M$
the generalized linear control system
\[ (\Lambda)\mbox{ \ \ \ \ \ \ \ \ \ \ \ \ \ \
}\dot{x}=\mathcal{F}(x)+\sum_{j=1}^{m}u_{j}Y_{j}(x) .
\]
Then we have the following theorem
\begin{theorem}
Suppose that the family of vector fields  $\Gamma=\{Y_{1},\ldots
,Y_{m}\}$ is transitive on  $M$. Then the action $(M,S_{\Lambda })$
is orbit equivalent to a semigroup action associated to a linear
control system on a homogeneous space.
\end{theorem}
\begin{proof}

Let $f:\tilde{M} \rightarrow M$ be the above differentiable
covering. Then from $\Lambda$ we define the following system in
$\tilde{M}$:
\[ \tilde{(\Lambda)} \mbox{ \ \ \ \ \ \ \ \ \ \ \ \ \ \
}\dot{x}=\tilde{\mathcal{F}}(x)+\sum_{j=1}^{m}u_{j}\tilde{Y}_{j}(x)
,
\]
where $\tilde{\mathcal{F}}$ and $\tilde{Y}_{j}$ are as defined
above.

Consider $\tilde{\Gamma}=\{\tilde{Y}_{1},\ldots ,\tilde{Y}_{m}\}$.
By definition of  $\tilde{Y}_{j}$ we have that the family
$\tilde{\Gamma}$ is complete and that $\mathcal{L}(\Gamma)$ is
isomorphic to  $\mathcal{L}(\tilde{\Gamma})$.

Note that $\tilde{\Gamma}$ is transitive. In fact,  every $f$-image
of orbit  is an orbit in $M$, moreover, the rank of $f$ is constant
in every orbit. As $\Gamma$ is transitive in $M$, the $\Gamma$-orbit
has the same dimension as $M$, and therefore, as $\tilde{\Gamma}$.
Then, the $\tilde{\Gamma}$-orbits in $\tilde{M}$ are submanifolds of
the same dimension of $\tilde{M}$. As $\tilde{M}$ is connected and
is the union of the $\tilde{\Gamma}$-orbits, it follows that exists
just one $\tilde{\Gamma}$-orbit. Therefore, $\tilde{\Gamma}$ is
transitive.

 Moreover, we have that
\[ [\tilde{\mathcal{F}},\tilde{Y}_{i}] \in \mathcal{L}(\tilde{\Gamma}),
\]
and as  $\mathcal{F}_{x_{0}}=0$ it follows that
$\tilde{\mathcal{F}}_{\tilde{x_{0}}}=0$ for all $\tilde{x_{0}} \in
f^{-1}(x_{0})$.

Consider the connected and simply connected Lie group $G$ with Lie
algebra  $\mathcal{L}(\Gamma)$ (and $\mathcal{L}(\tilde{\Gamma})$).

Then by Proposition  \ref{prop1}  we have   the actions
\[ G \times M \rightarrow M  \ \ \ \ \ \mbox{ and }  \ \ \ \ \ \ \ G \times \tilde{M} \rightarrow
\tilde{M} .
\]

Take  $\tilde{x_{0}} \in f^{-1}(x_{0})$ then we have the isotropy
subgroups
\[ H=\{g \in G; gx_{0}=x_{0} \} \ \ \ \ \mbox{ and } \ \ \ \ \ \ \ \ \
 \tilde{H}=\{g \in G; g\tilde{x}_{0}=\tilde{x}_{0}\} .
\]

Hence we have the  diffeomorphisms $\xi: M \rightarrow G/H$ given by
$\xi (g \cdot x_{0})=gH$ and $\tilde{\xi}: \tilde{M} \rightarrow
G/\tilde{H}$ with $\xi (g \cdot \tilde{x}_{0})=g\tilde{H}$,
 here $\cdot$ denote the action of $G$ on $M$ or $\tilde{M}$. As
$\tilde{M}$ is simply connected then $\tilde{H}$ is connected. As we
see in the demonstration of the previous result, we have that
$\tilde{\Lambda}$ is diffeomorphic to linear control system on
$G/\tilde{H}$. Now we describe this system on  $G/\tilde{H}$.

Consider $D:\mathcal{L}(\tilde{\Gamma}) \rightarrow
\mathcal{L}(\tilde{\Gamma})$ given by
$D(Y)=[\tilde{\mathcal{F}},Y]$, note that $D$ is derivation. Then
there exists a linear vector field  $\mathcal{X}$ on $\tilde{G}=G$
such that $D(Y)=[\mathcal{X},Y]$ for every $Y \in
\mathcal{L}(\tilde{\Gamma})$. Let $\tilde{\pi}:G \rightarrow
G/\tilde{H}$ be the canonical projection. By previous result, we
have that  $\tilde{\Lambda}$ is diffeomorphic to the following
linear control system in $G/\tilde{H}$:

\[ ({\Lambda}_{\tilde{\pi}}) \mbox{ \ \ \ \ \ \ \ \ \ \ \ \ \ \
}\dot{x}={\tilde{\pi}}_{\ast}(\mathcal{X})+\sum_{j=1}^{m}u_{j}{\tilde{\pi}}_{\ast}(\tilde{Y}_{j})
,
\]
where
${\tilde{\pi}}_{\ast}(\mathcal{X})=\tilde{\xi}_{\ast}(\tilde{\mathcal{F}})$
and
${\tilde{\pi}}_{\ast}(\tilde{Y}_{j})=\tilde{\xi}_{\ast}(\tilde{Y_{j}})$.

Note that $\tilde{\pi}(\mathcal{X})$ exists, i.e., o que é
$\tilde{H}$ is invariant by the flow of $\mathcal{X}$.

It is not difficult to see that $l:G/{\tilde{H}} \rightarrow G/H $
defined by $l(g{\tilde{H}})=gH$ is a differentiable covering.

Recall that we need to show that
 $\xi_{\ast}(\mathcal{F})$ is linear vector field on $G/H$
 and that $\xi_{\ast}(Y_{j})$ are right invariant vector field for
  $i=\{1, \ldots ,m\}$. As $Y_{j}$ is invariant we have that
  ${\pi}_{\ast}(Y_{j})$ is invariant  on $G/H$. By Lemma \ref{lema},
  we have that ${\pi}_{\ast}(Y_{j})=\xi_{\ast}(Y_{j})$, then
   $\xi_{\ast}(Y_{j})$ are invariants.

  The vector field  $\xi_{\ast}(\mathcal{F})$ is
linear if
 $\xi_{\ast }(\mathcal{F})$ is $\pi$-related with a linear vector field on
 $G$. Then, we first show that  $\pi_{\ast }(\mathcal{X})$ is   linear on $G/H$,
 i.e., $\mathcal{X}$ is $\pi$-related with $\pi_{\ast
}(\mathcal{X})$ in $G/H$. After this, we prove that $\xi_{\ast
}(\mathcal{F})=\pi_{\ast }(\mathcal{X})$.

First we note that $\xi_{\ast }(\mathcal{F})$ is null in
$H/\tilde{H}$. As $\tilde{\xi_{\ast
}}(\tilde{\mathcal{F}})=\tilde{\pi_{\ast }}(\mathcal{X})$ then
$\tilde{\pi_{\ast }}(\mathcal{X})$ is null in $H/\tilde{H}$. Then we
can prove that $H$ is invariant by the flow of $\mathcal{X}$. So
$(\mathcal{X}$ is $\pi$-related with the vector field
$\pi_{\ast}(\mathcal{X})$ on $G/H$.

Now we must prove that $\xi_{\ast }(\mathcal{F})=\pi_{\ast
}(\mathcal{X})$. As
${\tilde{\mathcal{F}}}_{\tilde{x}}=d(f{\mid}_{\tilde{U}})^{-1}{\mid}_{x}({\mathcal{F}}_{x})$
and  $\tilde{\xi}$ and $\xi$ are diffeomorphisms it follows that
\[\tilde{\xi _{\ast}}({\tilde{\mathcal{F}}}){\mid}_{g \tilde{H}}=d(l{\mid}_{\tilde{V}})^{-1}(\xi _{\ast}({\mathcal{F}}) \mid _{gH})\]
and
\[\tilde{\pi _{\ast}}({\tilde{\mathcal{X}}}){\mid}_{g \tilde{H}}=d(l{\mid}_{\tilde{W}})^{-1}(\pi _{\ast}({\mathcal{X}}) \mid _{gH})\]
then $\xi_{\ast }(\mathcal{F})=\pi_{\ast }(\mathcal{X})$.

\end{proof}


\begin{thebibliography}{99}

\bibitem{as} Agrachev A., Sachkov Y. Control
Theory from the Geometric Viewpoint. Berlin: Springer 2004.

\bibitem{aysm} Ayala V., San Martin L.A.B. Controllability
properties of a class of control systems on Lie groups. Lectures
Notes in Control and Inform. Sci. 2001; 258: 83---92.

\bibitem{ayti} Ayala V., Tirao J. Linear control systems on Lie
groups and local controllability. Differential geometry and control
(G. Ferreyra, R. Gardner, H. Hermes, and H. Sussmann, Eds.) Amer.
Math. Soc., Providence, Rhode Island. 1999; 47-64.

\bibitem{ColK00} Colonius F., Kliemann W. The Dynamics of
Control. Boston: Birk\-häuser  2000.

\bibitem{elliott} Elliott D.L. Bilinear control systems: Matrices in
action. New York: Springer 2009.

\bibitem{ellis} Ellis R. Cocycles in topological dynamics. Topology.
 1978; 17: 111--130.

\bibitem{jurdjevic} Jurdjevic V. Geometric  control theory. Cambridge:
Cambridge University Press 1997.

\bibitem{RoSaSacone} Rocio O.G., San Martin, L. A. B. and Santana A. J.  Invariant
cones and convex sets for bilinear control systems and parabolic
type of semigroups. J. Dyn. Control Syst. 2006; 12: 419-432.

\bibitem{RoSaVe} Rocio O.G., Santana A. J. and Verdi M. A. Semigroups of Affine
Groups, Controllability of Affine Systems and Affine Bilinear
Systems in ${\rm Sl}(2,{\Bbb R})\rtimes\mathbb{R}^{2}$. SIAM J.
Control Optim. 2009; 48: 1080-1088.

\bibitem{sachkov} Sachkov Y.L. Control theory on Lie groups. J.Math. Sci .Adv. Appl. 2009; 156:
381-439.

\bibitem{josiney} Souza J.A. On limit behavior of skew-product
transformation semigroups. Math. Nachr. 2013; 287: 91-104.

\bibitem{jouan1} Jouan P. Equivalence of Control Systems with Linear
Systems on Groups and Homogeneous Spaces. ESAIM Control Optim. Calc.
Var. 2010; 16: 956-973.

\bibitem{markus} Markus L. Controllability of multi-trajectories on
Lie groups. Dynamical Systems and Turbulence. 1980; 898: 250--265.

\bibitem{SM2} San Martin L.A.B. Grupos de Lie. Unpublished, 2014.

\bibitem{san martin} San Martin L.A.B. Invariant Control Sets on Flag
Manifolds. Math. Control Signals Systems. 1993; 6: 41-61.

\bibitem{san martin e tonelli} San Martin L.A.B., Tonelli P.A. Semigroup
Actions on Homogeneous Spaces. Semigroup Forum 1995; 50: 59-88.

\end{thebibliography}
\end{document}